\documentclass[12pt]{amsart}

\usepackage{accents}
\usepackage{appendix}
\usepackage{amsfonts}
\usepackage{amsmath}
\usepackage{amssymb}	
\usepackage{amsthm,bm}
\usepackage{array,booktabs,multirow}
\usepackage{braket}
\usepackage{centernot}
\usepackage{cite}
\usepackage{comment}
\usepackage{dsfont}
\usepackage{mathalfa}
\usepackage[shortlabels]{enumitem} 
\usepackage{etoolbox}
\usepackage{float}
\usepackage[hang, flushmargin]{footmisc}
\usepackage{latexsym}
\usepackage{lipsum}
\usepackage{needspace}
\usepackage{tikz}
\usepackage{hyperref}
\usetikzlibrary{matrix,arrows}

\newcommand{\dbar}{\,\mathchar'26\mkern-12mu d \hspace{0.06em}}

\theoremstyle{plain}
\newtheorem{thm}{Theorem}[section]

\newtheorem{lem}[thm]{Lemma}
\newtheorem{prop}[thm]{Proposition}

\theoremstyle{definition}

\newtheorem{defn}[thm]{Definition}

\theoremstyle{remark}

\AtBeginEnvironment{lem}{\Needspace{3\baselineskip}}%
\AtBeginEnvironment{thm}{\Needspace{3\baselineskip}}
\AtBeginEnvironment{prop}{\Needspace{3\baselineskip}}
\AtBeginEnvironment{cor}{\Needspace{3\baselineskip}}

\setlist[enumerate,1]{leftmargin=2em}

\def\F{\mathbb F}

\def\Z{\mathbb Z}
\def\U{U_q(\mathfrak{sl}_2)}

\title[Finite-dimensional irreducible $\triangle_q$-modules at roots of $1$]{Finite-dimensional irreducible modules of the universal Askey--Wilson algebra at roots of unity}

\author{Hau-Wen Huang}
\address{
Department of Mathematics\\
National Central University\\
Chung-Li 32001 Taiwan
}
\email{hauwenh@math.ncu.edu.tw}

\begin{document}

\begin{abstract}
Let $\F$ denote an algebraically closed field and assume that $q\in \F$ is a primitive $d^{\rm \, th}$ root of unity with $d\not=1,2,4$. The universal Askey--Wilson algebra $\triangle_q$ is a unital associative $\F$-algebra defined by generators and relations. The generators are $A,B,C$ and the relations assert that each of 
\begin{gather*}
A+\frac{qBC-q^{-1}CB}{q^2-q^{-2}},
\qquad 
B+\frac{qCA-q^{-1}AC}{q^2-q^{-2}},
\qquad 
C+\frac{qAB-q^{-1}BA}{q^2-q^{-2}}
\qquad 
\end{gather*}
commutes with $A,B,C$. We show that every finite-dimensional irreducible $\triangle_q$-module is of dimension less than or equal to 
$$
\left\{
\begin{array}{ll}
d 
\quad &\hbox{if $d$ is odd};
\\
d/2
\quad &\hbox{if $d$ is even}.
\end{array}
\right.
$$
Moreover we provide an example to show that the bound is tight. 
\end{abstract}

\maketitle



\section{Introduction}
Suppose that $\F$ is an algebraically closed field. 
In \cite{hidden_sym} A. Zhedanov introduced the Askey--Wilson algebras to link the Askey--Wilson polynomials \cite{ask85} and  representation theory. The algebras are  unital associative $\F$-algebras defined by generators and relations. The relations are slightly complicated, which involve a nonzero parameter $q$ and five additional parameters. To resolve the objection, P. Terwilliger proposed a central extension of the main confluence of the Askey--Wilson algebras defined as follows:

\begin{defn}
[Definition 1.2, \cite{uaw2011}]
\label{defn:AW}
Let $q$ denote a nonzero scalar in $\F$ with $q^4\not=1$. The {\it universal Askey--Wilson algebra} $\triangle_q$ is a unital associative $\F$-algebra generated by $A,B,C$ and the relations assert that each of 
\begin{gather}\label{alph&beta&gam}
A+\frac{qBC-q^{-1}CB}{q^2-q^{-2}},
\qquad 
B+\frac{qCA-q^{-1}AC}{q^2-q^{-2}},
\qquad 
C+\frac{qAB-q^{-1}BA}{q^2-q^{-2}}
\qquad 
\end{gather}
is central in $\triangle_q$.
\end{defn} 
The Askey--Wilson algebras and the universal Askey--Wilson algebra $\triangle_q$ are related to the integrable systems  \cite{zhedanov2008,Aneva2008}, the quantum harmonic oscillator \cite{Higgs2019,zhedanov2017},  the quantum algebra $\U$ \cite{Huang:RW,uaw&equit2011,Huang:CG,zhedanov1992,zhedanov1993}, the $q$-Onsager algebra \cite{uaw2011,ter2018}, the double affine Hecke algebras of type $(C_1^\vee, C_1)$ \cite{daha&LP,koo07,koo08,daha&Z3,DAHA2013,nilDAHA&qHahn:2016,nilDAHA&qHahn:2018}, the distance-regular graphs\cite{qRacahDRG,Tanaka:2018} and the Leonard pairs \cite{Huang:2012,Huang:DAHA&LT}. The classification of finite-dimensional irreducible $\triangle_q$-modules for $q$ not a root of unity was given in  \cite{Huang:2015} and has found connections to the representation theory of related algebras \cite{Huang:RW,Huang:AW&DAHAmodule,Huang:BImodule,
Huang:DAHAmodule,SH:2019-1}.  However the research on the representation theory of $\triangle_q$ at roots of unity is rarely seen.

Throughout the rest of this paper, 
suppose that $q$ is a primitive $d^{\rm \, th}$ root of unity with $d\not=1,2,4$ and set 
\begin{gather*}
\dbar 
=\left\{
\begin{array}{ll}
d 
\quad &\hbox{if $d$ is odd};
\\
d/2
\quad &\hbox{if $d$ is even}.
\end{array}
\right.
\end{gather*}
Note that $\dbar$ is the order of the root of unity $q^2$. The intention of this paper is to prove the following result:

\begin{thm}\label{thm:dimension}
If $V$ is a finite-dimensional irreducible $\triangle_q$-module then 
the dimension of $V$ is less than or equal to $\dbar$.
\end{thm}

\noindent Given any nonzero $a,b,c\in \F$ and any integer $n\geq 0$, an $(n+1)$-dimensional $\triangle_q$-module $V_n(a,b,c)$ is constructed in \cite[\S 4.1]{Huang:2015}. According to \cite[Theorem 4.4]{Huang:2015} the $(n+1)$-dimensional $\triangle_q$-module $V_n(a,b,c)$ is irreducible if and only if $0\leq n\leq \dbar-1$ and 
\begin{gather}\label{irreducible}
abc, a^{-1}bc, ab^{-1}c, abc^{-1}\not\in
\{q^{2i-n-1}\,|\,i=1,2,\ldots,n\}.
\end{gather}
Since an algebraically closed field is infinite, for all $0\leq n\leq \dbar-1$ there are infinitely many nonzero $a,b,c\in \F$ satisfying the condition (\ref{irreducible}). This establishes the existence of irreducible $\triangle_q$-modules with dimensions between $1$ to $\dbar$. Therefore the bound given in Theorem \ref{thm:dimension} is sharp.

The organization of this paper is as follows: In \S\ref{s:result} we present the preliminaries on $\triangle_q$, especially the link of the $q$-Racah sequences to the center of $\triangle_q$. 
In \S\ref{s:qRacah} we display some properties of $q$-Racah sequences and classify the $q$-Racah sequences into five types. 
In \S\ref{s:decomposition} we decompose the finite-dimensional irreducible $\triangle_q$-modules with respect to the corresponding $q$-Racah sequences. In \S\ref{s:KTE} we investigate the actions of several operators on the finite-dimensional irreducible $\triangle_q$-modules. In \S\ref{s:proof} we give a proof for Theorem \ref{thm:dimension}.

\section{Preliminaries}\label{s:result}

For notational convenience, we set $\alpha,\beta,\gamma$ to be the multiplication of  (\ref{alph&beta&gam}) by $q+q^{-1}$. In other words
\begin{align}
\frac{\alpha}{q+q^{-1}}
&=
A+\frac{qBC-q^{-1}CB}{q^2-q^{-2}},
\label{e:alpha}
\\
\frac{\beta}{q+q^{-1}}
&=
B+
\frac{qCA-q^{-1}AC}{q^2-q^{-2}},
\label{e:beta}
\\
\frac{\gamma}{q+q^{-1}}
&=
C+
\frac{qAB-q^{-1}BA}{q^2-q^{-2}}.
\label{e:gamma}
\end{align}
By Definition \ref{defn:AW} each of $\alpha,\beta,\gamma$ is a central element of $\triangle_q$.

\begin{lem}\label{lem:ABc}
The algebra $\triangle_q$ is generated by $A,B,\gamma$. Moreover 
\begin{align}
\alpha &= \frac{B^2A-(q^2+q^{-2})BAB+AB^2+(q^2-q^{-2})^2A+(q-q^{-1})^2B\gamma}{(q-q^{-1})(q^2-q^{-2})}, \label{e:a->ABc}
\\
\beta &= \frac{A^2B-(q^2+q^{-2})ABA+BA^2+(q^2-q^{-2})^2B+(q-q^{-1})^2A\gamma}{(q-q^{-1})(q^2-q^{-2})}. \label{e:b->ABc}
\end{align}
\end{lem}
\begin{proof}
By (\ref{e:gamma}) the element $C$ can be expressed in terms of $A,B,\gamma$ as follows:  
\begin{gather}
C =\frac{\gamma}{q+q^{-1}}-\frac{q A B-q^{-1} B A}{q^2-q^{-2}}.
\label{e:C->ABc} 
\end{gather}
Therefore $A,B,\gamma$ form a set of generators of $\triangle_q$. 
To get (\ref{e:a->ABc}) and (\ref{e:b->ABc}) substitute (\ref{e:C->ABc}) into (\ref{e:alpha}) and (\ref{e:beta}), respectively.
\end{proof}

Recall from \cite{hidden_sym,uaw2011} that the celebrated central element 
$$
\Omega=q ABC+q^2 A^2+q^{-2} B^2+q^2 C^2-q A\alpha-q^{-1} B\beta-q C \gamma
$$
is called the {\it Casimir element} of $\triangle_q$.

\begin{lem}[Theorem 7.5, \cite{uaw2011}]\label{lem:UAW_basis}
The elements 
\begin{gather*}
A^i B^j C^k \Omega^\ell \alpha^r \beta^s \gamma^t
\qquad 
\hbox{for all integers $i,j,k,\ell,r,s,t\geq 0$ with $ijk=0$}
\end{gather*}
are an $\F$-basis for $\triangle_q$.
\end{lem}

For each integer $n\geq 0$ we define the polynomial
$$
T_n(x) = 
\sum_{i=0}^{\lfloor n/2 \rfloor} 
(-1)^i
\left(
{n-i \choose i}+{n-i-1\choose i-1}
\right) x^{n-2 i}.
$$
Here we set ${-1\choose -1}=1$  and ${n\choose -1}=0$ for all $n\geq 0$.
Note that $\frac{1}{2}T_n(2x)$ is the $n$th Chebyshev polynomial of the first kind for all $n\geq 0$ \cite[\S 9.8.2]{Koe2010}.

\begin{lem}[Theorem 3.2, \cite{centerAW:2016}]\label{lem:center}
$T_{\dbar}(A), T_{\dbar}(B), T_{\dbar}(C)$ are central in $\triangle_q$.
\end{lem}

Let $\Z$ denote the additive group of integers.  
Let $\{\theta_i\}_{i\in \Z/\dbar \Z}$ denote a sequence in $\F$ of the form 
\begin{gather}\label{qRacah_seq}
\theta_i=aq^{-2i}+a^{-1} q^{2i}
\qquad 
\hbox{for all $i\in \Z/\dbar \Z$},
\end{gather}
where $a$ is a nonzero scalar in $\F$. We call the sequence $\{\theta_i\}_{i\in \Z/\dbar \Z}$ a {\it $q$-Racah sequence}.
It follows from  \cite[Equation 17]{centerAW:2016} that
$$
T_{\dbar}(x)=\prod_{i\in \Z/\dbar \Z}(x-\theta_i)+a^{\dbar}+a^{-\dbar}.
$$
This, combined with Lemma \ref{lem:center}, yields that 

\begin{lem}
[Proposition 3.5, \cite{centerAW:2016}]
\label{lem:q-Racah&center}
 If $\{\theta_i\}_{i\in \Z/\dbar \Z}$ is a $q$-Racah sequence then each of 
$$
\prod_{i\in \Z/\dbar \Z} (A-\theta_i),
\quad 
\prod_{i\in \Z/\dbar \Z} (B-\theta_i),
\quad 
\prod_{i\in \Z/\dbar \Z} (C-\theta_i)
$$
is central in $\triangle_q$.
\end{lem}

\begin{lem}\label{lem:basis2_UAW}
The elements 
\begin{gather*}
A^{i_0} B^{j_0} C^{k_0}
T_{\dbar} (A)^{i_1} T_{\dbar} (B)^{j_1} T_{\dbar} (C)^{k_1}
\Omega^\ell
\alpha^{r} \beta^{s} \gamma^{t}
\end{gather*}
for all $i_0,j_0,k_0\in\{0,1,\ldots,\dbar-1\}$ and all integers $i_1,j_1,k_1,\ell,r,s,t \geq 0$ with $(i_0+i_1)(j_0+j_1)(k_0+k_1)=0$ form an $\F$-basis for $\triangle_q$.
\end{lem}
\begin{proof}
Combine  Lemmas \ref{lem:UAW_basis} and \ref{lem:center}.
\end{proof}

\begin{lem}\label{lem:Schur}
If $V$ is a finite-dimensional irreducible $\triangle_q$-module then each central element of $\triangle_q$ acts on $V$ as scalar multiplication. 
\end{lem}
\begin{proof}
Immediate from Schur's lemma.
\end{proof}

\begin{prop}\label{prop:<2d}
If $V$ is a finite-dimensional irreducible $\triangle_q$-module then 
the dimension of $V$ is less than or equal to $\sqrt{3\dbar^2-3\dbar+1}$.
\end{prop}
\begin{proof}
Let $\pi:\triangle_q\to {\rm End}(V)$ denote the corresponding representation. According to Burnside's theorem on matrix algebras \cite{burnside1980}, the representation $\pi$ is onto. By Lemma \ref{lem:Schur} the images of 
$T_{\dbar} (A), T_{\dbar} (B), T_{\dbar} (C), \alpha, \beta, \gamma,\Omega$ under $\pi$ are scalar multiples of the identity. Combined with Lemma \ref{lem:basis2_UAW}, it follows that ${\rm End}(V)$ is spanned by 
\begin{gather}\label{e:spanEndV}
\pi(A)^{i_0} \pi(B)^{j_0} \pi(C)^{k_0}
\quad 
\hbox{for all $i_0,j_0,k_0\in\{0,1,\ldots,\dbar-1\}$ with $i_0j_0k_0=0$}
\end{gather}
over $\F$. The number of (\ref{e:spanEndV}) is 
$\dbar^3-(\dbar-1)^3=3\dbar^2-3\dbar+1$. 
Therefore 
$$
\dim V=\sqrt{\dim {\rm End}(V)}\leq \sqrt{3\dbar^2-3\dbar+1}.
$$ 
The proposition follows.
\end{proof}

\section{The $q$-Racah sequences}\label{s:qRacah}

In this section we show some properties of the $q$-Racah sequences and divide the $q$-Racah sequences into five types.

\begin{lem}\label{lem:theta}
Let $\{\theta_i\}_{i\in \Z/\dbar \Z}$ denote a $q$-Racah sequence. 
For all $i\in \Z/\dbar \Z$ the following {\rm (i)}, {\rm (ii)} hold:
\begin{enumerate}
\item $\theta_{i-1}+\theta_{i+1}=(q^{2}+q^{-2})\theta_{i}$.

\item $\theta_{i-1}\theta_{i+1}=\theta_i^2+(q^2-q^{-2})^2$.
\end{enumerate}
\end{lem}
\begin{proof}
It is straightforward to verify (i), (ii) by using (\ref{qRacah_seq}).
\end{proof}

\begin{lem}\label{lem:theta^2}
Let $\{\theta_i\}_{i\in \Z/\dbar \Z}$ denote a $q$-Racah sequence. 
For each $i\in \Z/\dbar \Z$ the following are equivalent: 
\begin{enumerate}
\item $\theta_{i-1}=\theta_{i+1}$.

\item $(q^2+q^{-2})\theta_{i}=2\theta_{i-1}$. 

\item $(q^2+q^{-2})\theta_{i}=2\theta_{i+1}$. 

\item $\theta_i\in\{\pm 2\}$.
\end{enumerate}
\end{lem}
\begin{proof}
Since $\{\theta_i\}_{i\in \Z/\dbar \Z}$ is a $q$-Racah sequence  there exists a nonzero scalar $a\in \F$ such that 
\begin{align}\label{e:pm2}
(\theta_{i-1},\theta_i,\theta_{i+1})=(a q^2+a^{-1} q^{-2},a+a^{-1},a q^{-2}+a^{-1} q^{2}).
\end{align}
Using (\ref{e:pm2}) yields that each of (i)--(iv) holds if and only if $a\in \{\pm 1\}$. Therefore (i)--(iv) are equivalent. 
\end{proof}

Given any two sequences $\{\theta_i\}_{i\in \Z/\dbar \Z}$ and $\{\theta'_i\}_{i\in \Z/\dbar \Z}$ in $\F$, we say that $\{\theta_i\}_{i\in \Z/\dbar \Z}$ is {\it congruent} to $\{\theta'_i\}_{i\in \Z/\dbar \Z}$ if there exists an element $j\in \Z/\dbar \Z$ such that 
$$
\theta_i=\theta'_{i+j}
\qquad 
\hbox{for all $i\in \Z/\dbar \Z$}.
$$
Note that the congruence relation is an equivalence relation and the set of $q$-Racah sequences is closed under the congruence relation.

\begin{defn}\label{defn:DEO}
\begin{enumerate}
\item A $q$-Racah sequence $\{\theta_i\}_{i\in \Z/\dbar\Z}$ is said to be of {\it type $D$} if there exists a nonzero $a\in \F$ with $a^2 \not\in \{q^{2i}\,|\,i\in \Z/\dbar \Z\}$ such that 
$$
\theta_i=a q^{-2i}+a^{-1} q^{2i}
\qquad 
\hbox{for all $i\in  \Z/\dbar\Z$}.
$$

\item A $q$-Racah sequence $\{\theta_i\}_{i\in \Z/\dbar\Z}$ is said to be of {\it type $O(2)$} if $\dbar$ is odd and 
$$
\theta_i=q^{2i}+q^{-2i}
\qquad 
\hbox{for all $i\in  \Z/\dbar\Z$}.
$$

\item A $q$-Racah sequence $\{\theta_i\}_{i\in \Z/\dbar\Z}$ is said to be of {\it type $O(-2)$} if $\dbar$ is odd and 
$$
\theta_i=-q^{2i}-q^{-2i}
\qquad 
\hbox{for all $i\in  \Z/\dbar\Z$}.
$$

\item A $q$-Racah sequence $\{\theta_i\}_{i\in \Z/\dbar\Z}$ is said to be of {\it type $E(2)$} if $\dbar$ is even and 
$$
\theta_i=q^{2i}+q^{-2i}
\qquad 
\hbox{for all $i\in  \Z/\dbar\Z$}.
$$

\item A $q$-Racah sequence $\{\theta_i\}_{i\in \Z/\dbar\Z}$ is said to be of {\it type $E(q+q^{-1})$} if $\dbar$ is even and 
$$
\theta_i=q^{2i-1}+q^{1-2i}
\qquad 
\hbox{for all $i\in  \Z/\dbar\Z$}.
$$

\end{enumerate}
\end{defn}

\begin{prop}\label{prop:congruence}
Let $\{\theta_i\}_{i\in \Z/\dbar \Z}$ denote a $q$-Racah sequence. If $\{\theta_i\}_{i\in \Z/\dbar \Z}$ is not of type $D$, then $\{\theta_i\}_{i\in \Z/\dbar \Z}$ is congruent to the $q$-Racah sequence of type $O(2),O(-2),E(2)$ or $E(q+q^{-1})$.
\end{prop}
\begin{proof}
Since $\{\theta_i\}_{i\in \Z/\dbar \Z}$ is not of type $D$ it follows from Definition \ref{defn:DEO}(i) that $\{\theta_i\}_{i\in \Z/\dbar \Z}$ is congruent to one of the following sequences:
\begin{gather}
\{q^{2i}+q^{-2i}\}_{i\in \Z/\dbar \Z},
\label{c1}
\\
\{-q^{2i}-q^{-2i}\}_{i\in \Z/\dbar \Z},
\label{c2}
\\
\{q^{2i-1}+q^{1-2i}\}_{i\in \Z/\dbar \Z},
\label{c3}
\\
\{-q^{2i-1}-q^{1-2i}\}_{i\in \Z/\dbar \Z}.
\label{c4}
\end{gather} 
We divide the argument into the two cases: (i) $\dbar$ is odd; (ii) $\dbar$ is even.

(i): By Definition \ref{defn:DEO}(ii) the sequence (\ref{c1}) is of type $O(2)$. By Definition \ref{defn:DEO}(iii) the sequence (\ref{c2}) is of type $O(-2)$.  Let $\{\theta_i'\}_{i\in \Z/\dbar \Z}$ denote the sequence (\ref{c3}). Observe that 
$$
\theta_{i+\frac{\dbar+1}{2}}'
=
q^{2i+\dbar}+q^{-\dbar-2i}
\qquad 
\hbox{for all $i\in \Z/\dbar \Z$}.
$$
Since $\dbar$ is the order of $q^2$, the scalar $q^{\dbar}\in\{\pm 1\}$. 
It follows that (\ref{c3}) is congruent to the $q$-Racah sequence of type $O(2)$ if $q^{\dbar}=1$; the $q$-Racah sequence of type $O(-2)$ if $q^{\dbar}=-1$.
Similarly (\ref{c4}) is congruent to the $q$-Racah sequence of type $O(-2)$ if $q^{\dbar}=1$; the $q$-Racah sequence of type $O(2)$ if $q^{\dbar}=-1$. 
Therefore $\{\theta_i\}_{i\in \Z/\dbar \Z}$ is congruent to the $q$-Racah sequence of type $O(2)$ or $O(-2)$. 

(ii): By Definition \ref{defn:DEO}(iv) the $q$-Racah sequence (\ref{c1}) is of type $E(2)$. By Definition \ref{defn:DEO}(v) the $q$-Racah sequence (\ref{c3}) is of type $E(q+q^{-1})$. Let $\{\theta_i'\}_{i\in \Z/\dbar \Z}$ denote the sequence (\ref{c2}).
Observe that $q^{\dbar}=-1$ whenever $\dbar$ is even.  Hence 
$\theta_i'=q^{2i-\dbar}+q^{\dbar-2i}$ for all $i\in \Z/\dbar \Z$. 
It follows that  
$$
\theta_{i+\frac{\dbar}{2}}'=q^{2i}+q^{-2i}
\qquad 
\hbox{for all $i\in \Z/\dbar \Z$}.
$$
This implies that $\{\theta_i'\}_{i\in \Z/\dbar \Z}$ is congruent to the $q$-Racah sequence of type $E(2)$. By a similar argument  the sequence (\ref{c4}) is congruent to the $q$-Racah sequence of type $E(q+q^{-1})$. Therefore $\{\theta_i\}_{i\in \Z/\dbar \Z}$ is congruent to the $q$-Racah sequence of type $E(2)$ or $E(q+q^{-1})$. The proposition follows.
\end{proof}

\begin{prop}\label{prop:multi_DEO}
\begin{enumerate}
\item If $\{\theta_i\}_{i\in \Z/\dbar \Z}$ is a $q$-Racah sequence of type $D$ then the scalars $\{\theta_i\}_{i\in \Z/\dbar \Z}$ are mutually distinct.

\item If $\{\theta_i\}_{i\in \Z/\dbar \Z}$ is the $q$-Racah sequence of type $E(2), O(2)$ or $O(-2)$ then $\theta_i=\theta_j$ if and only if $i\in\{j,-j\}$ for any $i,j\in \Z/\dbar \Z$.

\item If $\{\theta_i\}_{i\in \Z/\dbar \Z}$ is the $q$-Racah sequence of type $E(q+q^{-1})$ then  $\theta_i=\theta_j$ if and only if $i\in\{j,1-j\}$ for any $i,j\in \Z/\dbar \Z$.
\end{enumerate}
\end{prop}
\begin{proof}
(i): Using Definition \ref{defn:DEO}(i) yields that $\theta_i-\theta_j$ is equal to 
\begin{gather}\label{e:mult_D}
(q^{i-j}-q^{j-i})(a^{-1} q^{i+j}-a q^{-i-j})
\end{gather}
for $i,j\in \Z/\dbar\Z$. Since the root of unity $q^2$ is of order $\dbar$, the first multiplicand  of (\ref{e:mult_D}) is zero if and only if $i=j$. Since $a^2 \not\in \{q^{2i}\,|\,i\in \Z/\dbar \Z\}$ the second multiplicand of (\ref{e:mult_D}) is nonzero. The statement (i) follows.

(ii): Using Definition \ref{defn:DEO}(ii)--(iv) yields that 
$\theta_i-\theta_j$ is equal to 
\begin{gather}\label{e:mult_E(2)}
\pm (q^{i-j}-q^{j-i})(q^{i+j}-q^{-i-j})
\end{gather}
for $i,j\in \Z/\dbar\Z$. Observe that the first multiplicand  of (\ref{e:mult_E(2)}) is zero if and only if $i=j$; the second multiplicand  of (\ref{e:mult_E(2)}) is zero if and only if $i=-j$. The statement (ii) follows.

(iii): Using Definition \ref{defn:DEO}(v) yields that 
$\theta_i-\theta_j$ is equal to 
\begin{gather}\label{e:mult_E(q)}
(q^{i-j}-q^{j-i})(q^{i+j-1}-q^{1-i-j})
\end{gather}
for $i,j\in \Z/\dbar\Z$.  Observe that the first multiplicand  of (\ref{e:mult_E(q)}) is zero if and only if $i=j$; the second multiplicand  of (\ref{e:mult_E(q)}) is zero if and only if $i=1-j$. The statement (iii) follows.
\end{proof}

\section{The decompositions of finite-dimensional irreducible $\triangle_q$-modules}\label{s:decomposition}

For the rest of this paper, we adopt the following notations: Let $V$ denote a finite-dimensional irreducible $\triangle_q$-module. For all integers $n\geq 1$ and all $\theta\in \F$, let 
$$
V^{(n)}(\theta)=\{v\in V\,|\, (B-\theta)^n v=0\}
$$
and we simply write $V(\theta)= V^{(1)}(\theta)$. Note that $V^{(n)}(\theta)$ is $B$-invariant for all $n\geq 1$ and all $\theta\in \F$. Let $\{\theta_i\}_{i\in \Z/\dbar\Z}$ stand for a $q$-Racah sequence such that $V(\theta_i)\not=\{0\}$ for some $i\in \Z/\dbar \Z$. Since $\F$ is algebraically closed, the sequence $\{\theta_i\}_{i\in \Z/\dbar\Z}$ exists.

\begin{lem}\label{lem:vanish}
The element
\begin{gather}\label{vanish}
\prod\limits_{i\in \Z/\dbar \Z} (B-\theta_i)
\end{gather}
vanishes on $V$.
\end{lem}
\begin{proof}
Recall from Lemma \ref{lem:q-Racah&center} that (\ref{vanish}) is central in $\triangle_q$. Combined with Lemma \ref{lem:Schur} the element (\ref{vanish}) acts on $V$ as scalar multiplication. By the choice of $\{\theta_i\}_{i\in \Z/\dbar \Z}$ there exists an $i\in \Z/\dbar \Z$ such that $V(\theta_i)\not=\{0\}$. For any nonzero $v\in V(\theta_i)$, the operator (\ref{vanish}) sends $v$ to zero. It follows that (\ref{vanish}) acts on $V$ as scalar multiplication by zero. The lemma follows.
\end{proof}

In the following we describe the decomposition of $V$ with respect to the types of $\{\theta_i\}_{i\in \Z/\dbar\Z}$.

\begin{prop}\label{prop:decomposition}
\begin{enumerate}
\item If $\{\theta_i\}_{i\in \Z/\dbar \Z}$ is of type $D$ then 
$$
V=\bigoplus\limits_{i=0}^{\dbar-1} V(\theta_i).
$$

\item If $\{\theta_i\}_{i\in \Z/\dbar \Z}$ is of type $O(2)$ or $O(-2)$ then 
$$
V=V(\theta_0)\oplus \bigoplus\limits_{i=1}^{\frac{\dbar-1}{2}} V^{(2)}(\theta_i).
$$

\item If $\{\theta_i\}_{i\in \Z/\dbar \Z}$ is of type $E(2)$ then 
$$
V=V(\theta_0)\oplus 
\bigoplus\limits_{i=1}^{\frac{\dbar}{2}-1} V^{(2)}(\theta_i)
\oplus V(\theta_\frac{\dbar}{2}).
$$

\item If $\{\theta_i\}_{i\in \Z/\dbar \Z}$ is of type $E(q+q^{-1})$ then 
$$
V=\bigoplus\limits_{i=1}^{\frac{\dbar}{2}} V^{(2)}(\theta_i).
$$ 
\end{enumerate}
\end{prop}
\begin{proof}
(i): By Proposition \ref{prop:multi_DEO}(i) the scalars $\{\theta_i\}_{i\in \Z/\dbar \Z}$ are mutually distinct. Combined with Lemma \ref{lem:vanish} this implies (i).

(ii): By Proposition \ref{prop:multi_DEO}(ii) and since $\dbar$ is odd, the scalars $\theta_0$ and $\{\theta_i\}_{i=1}^{\frac{\dbar-1}{2}}$ are all distinct values in the sequence $\{\theta_i\}_{i\in \Z/\dbar \Z}$. Moreover $\theta_0$ is of multiplicity $1$ and each of $\{\theta_i\}_{i=1}^{\frac{\dbar-1}{2}}$ is of multiplicity $2$ in $\{\theta_i\}_{i\in \Z/\dbar \Z}$. Combined with Lemma \ref{lem:vanish} the statement (ii) follows.

(iii), (iv): Similar to the proof of (ii).
\end{proof}

\section{The operators associated with $q$-Racah sequences}\label{s:KTE}

Define
\begin{align*}
K_i^{(n)}&=
(B-\theta_{i-1})^n
(B-\theta_{i+1})^n A,
\\
T_i^{(n)}&=
(B-\theta_{i-1})^n
(B-\theta_i)^n
(B-\theta_{i+1})^n A,
\\
E_i^{(n)}&=
(B-\theta_i)^n
(B-\theta_{i+1})^n
A
\end{align*}
for all $i\in \Z/\dbar \Z$ and $n=1,2$.
In this section we investigate the actions of the above operators on $V$. 
For short we write $K_i=K_i^{(1)}, T_i=T_i^{(1)}, E_i=E_i^{(1)}$ for all $i\in \Z/\dbar \Z$.

\begin{prop}\label{prop:EKT}
For all $i\in \Z/\dbar \Z$ the following hold:
\begin{enumerate}
\item $
K_i V(\theta_i)\subseteq V(\theta_i) 
$
and $K_i$ acts on $V(\theta_i)$ as scalar multiplication.

\item $
T_i
$ vanishes on $V(\theta_i)$. 

\item $
E_i V(\theta_i)\subseteq V(\theta_{i-1}).
$
\end{enumerate}
\end{prop}
\begin{proof}
(i): Pick any $v\in V(\theta_i)$. Applying $v$ to either side of (\ref{e:a->ABc}) yields  that 
\begin{gather}\label{e:EF}
(B^2-(q^2+q^{-2})\theta_i B+\theta_i^2+(q^2-q^{-2})^2) A v=(q-q^{-1})(q^2-q^{-2})\alpha v-(q-q^{-1})^2\theta_i \gamma v.
\end{gather}
By Lemma \ref{lem:theta} the left-hand side of (\ref{e:EF}) can be factored into 
$$
(B-\theta_{i-1})(B-\theta_{i+1})Av=K_i v.
$$
By Lemma \ref{lem:Schur} the right-hand side of (\ref{e:EF}) is a multiple of $v$ by a scalar which is independent of the choice of $v\in V(\theta_i)$. Therefore (i) follows.

(ii): Since $T_i=(B-\theta_i) K_i$ the statement (ii) is immediate from (i).

(iii): Since $T_i=(B-\theta_{i-1}) E_i$ the statement (iii) is immediate from (ii).
\end{proof}

\begin{lem}\label{lem:A}
For all $i\in \Z/\dbar \Z$ the $\F$-subspace $A V(\theta_i)$ of $V$ is contained in 
$$
\left\{
\begin{array}{ll}
V(\theta_{i-1})+V(\theta_i)+V(\theta_{i+1})
\quad 
&\hbox{if $\theta_{i-1},\theta_i,\theta_{i+1}$ are mutually distinct},
\\
V^{(2)}(\theta_i)+V(\theta_{i+1})
\quad 
&\hbox{if $\theta_i=\theta_{i-1}$},
\\
V^{(2)}(\theta_i)+V(\theta_{i-1})
\quad 
&\hbox{if $\theta_i=\theta_{i+1}$},
\\
V^{(2)}(\theta_{i-1})+V(\theta_i)
\quad 
&\hbox{if $\theta_{i-1}=\theta_{i+1}$}.
\end{array}
\right.
$$
\end{lem}
\begin{proof}
Immediate from Proposition \ref{prop:EKT}(ii).
\end{proof}

\begin{lem}\label{lem:KV(2)-MV} 
For all $i\in \Z/\dbar \Z$ the $\F$-subspace $V^{(2)}(\theta_i)$ of $V$ is invariant under 
\begin{gather}\label{K-BAB}
K_i -((q^2+q^{-2})B-2\theta_i) A (B-\theta_i).
\end{gather}
\end{lem}
\begin{proof}
Let $v\in V^{(2)}(\theta_i)$ be given. Applying $v$ to either side of (\ref{e:a->ABc}) and using Lemma \ref{lem:theta}, a direct calculation yields that the image of $v$ under (\ref{K-BAB}) is equal to 
\begin{align}\label{e:K2-pre}
(q-q^{-1})(q^2-q^{-2})\alpha v-(q-q^{-1})^2 B\gamma v.
\end{align}
By Lemma \ref{lem:Schur}, the vector (\ref{e:K2-pre}) is in $V^{(2)}(\theta_i)$. The lemma follows.
\end{proof}

\begin{prop}\label{prop:EKT2}
For all $i\in \Z/\dbar \Z$ the following hold:
\begin{enumerate}
\item 
$
K_i^{(2)} V^{(2)}(\theta_i)\subseteq 
V^{(2)}(\theta_i).
$

\item $T_i^{(2)}$ vanishes on $V^{(2)}(\theta_i)$.

\item $E_i^{(2)} V^{(2)}(\theta_i)\subseteq V^{(2)}(\theta_{i-1})$.
\end{enumerate}
\end{prop}
\begin{proof}
(i): 
From Lemma \ref{lem:KV(2)-MV} we see that 
\begin{gather*}
K_i  V^{(2)}(\theta_i)
\subseteq  V^{(2)}(\theta_i)+((q^2+q^{-2})B-2\theta_i) A V(\theta_i).
\end{gather*}
Applying Lemma \ref{lem:A} yields that 
$$
K_i  V^{(2)}(\theta_i)
\subseteq
\left\{
\begin{array}{ll}
V(\theta_{i-1})+V^{(2)}(\theta_i)+V(\theta_{i+1})
\quad 
&\hbox{if $\theta_{i-1},\theta_i,\theta_{i+1}$ are mutually distinct},
\\
V^{(2)}(\theta_i)+V(\theta_{i+1})
\quad 
&\hbox{if $\theta_i=\theta_{i-1}$},
\\
V^{(2)}(\theta_i)+V(\theta_{i-1})
\quad 
&\hbox{if $\theta_i=\theta_{i+1}$},
\\
V^{(2)}(\theta_{i-1})+V^{(2)}(\theta_i)
\quad 
&\hbox{if $\theta_{i-1}=\theta_{i+1}$}.
\end{array}
\right.
$$
Combined with $K_i^{(2)}=(B-\theta_{i-1})(B-\theta_{i+1}) K_i$ the statement (i) follows.

(ii): Since $T_i^{(2)}=(B-\theta_i)^2 K_i^{(2)}$ the statement (ii) is immediate from (i).

(iii): Since $T_i^{(2)}=(B-\theta_{i-1})^2 E_i^{(2)}$ the statement (iii) is immediate from (ii).
\end{proof}

\section{Proof of Theorem \ref{thm:dimension}}\label{s:proof}

We are devoted to proving Theorem \ref{thm:dimension} in the final section.

\begin{lem}\label{lem2:bidiagonal}
For all $i\in \Z/\dbar \Z$ the following are equivalent:
\begin{enumerate}
\item $
E_i|_{V(\theta_i)}
$
is not injective.

\item There exists
an eigenvector of 
$
(B-\theta_{i+1})A
$
in $V(\theta_i)$. 
\end{enumerate}
\end{lem}
\begin{proof}  
The implication from (ii) to (i) is trivial. We show the implication from (i) to (ii).

(i) $\Rightarrow$ (ii): First suppose that $\theta_{i-1}=\theta_i$.  Note that $E_i=K_i$ in this case. By (i) the operator $K_i$ vanishes at some nonzero vector of $V(\theta_i)$. Combined with Proposition  \ref{prop:EKT}(i) the element $K_i$ vanishes on $V(\theta_i)$. It follows that $V(\theta_i)$ is $(B-\theta_{i+1})A$-invariant. 
Since $\F$ is algebraically closed, there exists an eigenvector of $(B-\theta_{i+1})A$ in $V(\theta_i)$.

Next consider the case $\theta_{i-1}\not=\theta_i$.
By (i) there exists a nonzero $v\in V(\theta_i)$ such that 
$E_i v=0$.
It follows that
\begin{gather}\label{bidiagonal_1}
(B-\theta_{i+1}) A v\in V(\theta_i).
\end{gather}
Using (\ref{bidiagonal_1}) yields that 
$$
K_i v=(\theta_i-\theta_{i-1})(B-\theta_{i+1})A v.
$$
On the other hand, $K_i v$ is a scalar multiple of $v$ by Proposition \ref{prop:EKT}(i). Since $\theta_{i-1}\not=\theta_i$ it follows that $v$ is an eigenvector of 
$
(B-\theta_{i+1})A$. 

We have shown that the implication from (i) to (ii) is true in either case. The lemma follows. 
\end{proof}

\begin{lem}\label{lem:bidiagonal}
For any $i\in \Z/\dbar \Z$, 
if there exists an eigenvector $v$ of $(B-\theta_{i+1})A$ in $V(\theta_i)$ then
$$
(B-\theta_{i+j})A^j v
$$
is an $\F$-linear combination of $v, Av,\ldots, A^{j-1} v$ for all $j=0,1,\ldots, \dbar-1$.
\end{lem}
\begin{proof}
Since $v\in V(\theta_i)$ the statement is true for $j=0$. 
Since $v$ is an eigenvector of $(B-\theta_{i+1})A$ the statement is true for $j=1$.
Suppose that $j\geq 2$. 
Applying $A^{j-2}v$ to either side of (\ref{e:b->ABc}) we have 
\begin{align}\label{e:vj-2-B}
A^2 B A^{j-2}v
-(q^2+q^{-2}) ABA^{j-1}v
+
BA^j v
+(q^2-q^{-2})^2 BA^{j-2} v
\end{align}
is equal to 
\begin{align}\label{e:vj-2}
(q-q^{-1})(q^2-q^{-2}) A^{j-2} \beta v-(q-q^{-1})A^{j-1}\gamma  v.
\end{align}
By Lemma \ref{lem:Schur} the term (\ref{e:vj-2}) is an $\F$-linear combination of $A^{j-2} v$ and $A^{j-1} v$.
Applying induction hypothesis the term (\ref{e:vj-2-B}) is equal to 
\begin{gather}\label{e:vj-2-B'}
(\theta_{i+j-2}-(q^2+q^{-2})\theta_{i+j-1}+B) A^j v
\end{gather}
plus an $\F$-linear combination of $v,Av,\ldots,A^{j-1} v$. 
Using Lemma \ref{lem:theta}(i) yields that (\ref{e:vj-2-B'}) is equal to 
$(B-\theta_{i+j})A^j v$. Combining the above comments, the lemma follows.
\end{proof}

\begin{thm}\label{thm:dim_bidiag}
If there exists an $i\in \Z/\dbar \Z$ such that 
$
E_i|_{V(\theta_i)}
$
is not injective, then $V$ is of dimension less than or equal to $\dbar$.
\end{thm}
\begin{proof}
By Lemma \ref{lem2:bidiagonal} there exists an eigenvector $v$ of $(B-\theta_{i+1}) A$ in $V(\theta_i)$. 
Let $W$ denote the $\F$-subspace of $V$ spanned by 
$$
v, A v,\ldots, A^{\dbar-1} v.
$$  
Apparently $W$ is $\gamma$-invariant by Lemma \ref{lem:Schur}. 
It follows from Lemma \ref{lem:bidiagonal} that $BA^j v$ is an $\F$-linear combination of $v, Av,\ldots, A^j v$ for all $j=0,1,\ldots, \dbar-1$. Hence $W$ is $B$-invariant.
Since $T_{\dbar}(A)$ is central in $\triangle_q$ by Lemma \ref{lem:center}, it follows from Lemma \ref{lem:Schur} that $T_{\dbar}(A)v$ is a scalar multiple of $v$. Since $T_{\dbar}(x)$ is of degree $\dbar$, it follows that $A^{\dbar} v\in W$. Hence $W$ is $A$-invariant. Since $W$ is invariant under $A,B,\gamma$ it follows from Lemma \ref{lem:ABc} that $W$ is a $\triangle_q$-submodule of $V$. Since $v\not=0$ and 
the $\triangle_q$-module $V$ is irreducible, it follows that $V=W$. 
By construction $W$ is of dimension less than or equal to $\dbar$.   
The theorem follows.
\end{proof}

We now deal with the case that $E_i|_{V(\theta_i)}
$
is injective for each $i\in \Z/\dbar \Z$.

\begin{lem}\label{lem:cyclic_V(1)}
If 
$
E_i|_{V(\theta_i)}
$
is injective for each $i\in \Z/\dbar \Z$ then 
$$
V(\theta_i)
\qquad  
\hbox{for all $i\in \Z/\dbar \Z$}
$$ 
are of the same dimension. 
\end{lem}
\begin{proof}
By Proposition \ref{prop:EKT}(iii) and since $E_i|_{V(\theta_i)}$ is injective for all $i\in \Z/\dbar\Z$, it follows that
$$
\dim V(\theta_i)\leq  \dim V(\theta_{i-1})
\qquad 
\hbox{for all $i\in \Z/\dbar \Z$}.
$$
Since $\Z/\dbar \Z$ is a finite cyclic group this implies that $\dim V(\theta_i)$ are equal for all $i\in \Z/\dbar \Z$. The lemma follows.
\end{proof}

\begin{lem}\label{lem:cyclic_dimV(2)=2dimV(1)}
Suppose that 
$
E_i|_{V(\theta_i)}
$
is injective for each $i\in \Z/\dbar \Z$. If $\{\theta_i\}_{i\in \Z/\dbar \Z}$ is of type $O(2), O(-2), E(2)$ or $E(q+q^{-1})$, then 
\begin{gather}\label{dimV(2)=2dimV(1)}
\dim V^{(2)}(\theta_1)=2 \cdot \dim V(\theta_1).
\end{gather}
\end{lem}
\begin{proof}
Suppose that $\{\theta_i\}_{i\in \Z/\dbar \Z}$ is of type $O(2), O(-2)$ or $E(2)$. 
By Proposition \ref{prop:EKT}(iii) and since $\theta_{-1}=\theta_1$ by Proposition \ref{prop:multi_DEO}(ii), we have the map
\begin{gather}\label{E0}
E_0|_{V(\theta_0)}:V(\theta_0)\to V(\theta_1).
\end{gather}
We decompose (\ref{E0}) into a composition of $(B-\theta_0)A|_{V(\theta_0)}:V(\theta_0)\to V^{(2)}(\theta_1)$ followed by 
\begin{gather}\label{E0_2nd}
(B-\theta_1)|_{V^{(2)}(\theta_1)}:V^{(2)}(\theta_1)\to V(\theta_1).
\end{gather}
Since (\ref{E0}) is injective and by Lemma \ref{lem:cyclic_V(1)} it follows that (\ref{E0}) is an isomorphism. Hence (\ref{E0_2nd}) is surjective. Now, applying the rank-nullity theorem to (\ref{E0_2nd}) the equality (\ref{dimV(2)=2dimV(1)}) follows.

Suppose that $\{\theta_i\}_{i\in \Z/\dbar \Z}$ is of type $E(q+q^{-1})$.
By Proposition \ref{prop:EKT}(iii) and since $\theta_0=\theta_1$ by Proposition \ref{prop:multi_DEO}(iii), we have the map
\begin{gather}\label{E1}
E_1|_{V(\theta_1)}:V(\theta_1)\to V(\theta_1).
\end{gather}
We decompose the map (\ref{E1}) into a composition of $(B-\theta_2)A|_{V(\theta_1)}:V(\theta_1)\to V^{(2)}(\theta_1)$ followed by 
\begin{gather}\label{E1_2nd}
(B-\theta_1)|_{V^{(2)}(\theta_1)}:V^{(2)}(\theta_1)\to V(\theta_1).
\end{gather}
Since (\ref{E1}) is injective it follows that (\ref{E1}) is an isomorphism. Hence (\ref{E1_2nd}) is surjective. Now, applying the rank-nullity theorem to (\ref{E1_2nd}) the equality (\ref{dimV(2)=2dimV(1)}) follows.  
\end{proof}

\begin{lem}\label{lem:cyclic_E(2)V(1)}
For any $i\in \Z/\dbar \Z$ with $\theta_i\not\in\{\theta_{i-1},\pm2\}$, 
if 
$
E_i|_{V(\theta_i)}
$
is injective then 
$
E^{(2)}_i|_{V(\theta_i)}
$ is injective.
\end{lem}
\begin{proof}
Suppose that $v$ lies in the kernel of $E^{(2)}_i$ in $V(\theta_i)$. We show that $v=0$.
It follows from Proposition \ref{prop:EKT}(iii) that 
\begin{gather*}
E_i v\in V(\theta_{i-1}).
\end{gather*}
Applying $(B-\theta_i)(B-\theta_{i+1})$ to $E_i v$ yields that $E^{(2)}_i v$ is equal to the scalar multiple of $E_i v$ by  
\begin{gather}\label{e:scalar1}
(\theta_{i-1}-\theta_i) (\theta_{i-1}-\theta_{i+1}).
\end{gather}
Since $\theta_i\not\in\{\pm 2\}$ it follows from Lemma \ref{lem:theta^2} that $\theta_{i-1}\not=\theta_{i+1}$. Combined with $\theta_i\not=\theta_{i-1}$ the scalar  (\ref{e:scalar1}) is nonzero. Since $E^{(2)}_i v=0$ this forces that $E_i v=0$. By the assumption that $E_i|_{V_i(\theta_i)}$ is injective, the vector $v=0$.  
The lemma follows.
\end{proof}

\begin{lem}\label{lem:cyclic_E(2)}
For any $i\in \Z/\dbar \Z$ with $\theta_i\not\in\{ \theta_{i-2},\theta_{i-1},\pm2\}$, 
if 
$
E_i|_{V(\theta_i)}
$
is injective then 
$
E^{(2)}_i|_{V^{(2)}(\theta_i)}
$ is injective.
\end{lem}
\begin{proof}
 Suppose that $v$ lies in the kernel of $E^{(2)}_i$ in ${V^{(2)}(\theta_i)}$. In light of Lemma \ref{lem:cyclic_E(2)V(1)} it suffices to show that $v\in V(\theta_i)$. It follows from Lemma \ref{lem:KV(2)-MV} that 
\begin{gather}\label{e:Gv}
K_i v
\in ((q^2+q^{-2})B-2\theta_i)A(B-\theta_i)v+V^{(2)}(\theta_i).
\end{gather}
Applying $(B-\theta_{i+1})(B-\theta_i)^2$ to either side of (\ref{e:Gv}) yields that 
\begin{gather}\label{AE^2v}
(B-\theta_{i-1}) E^{(2)}_i v=((q^2+q^{-2})B-2\theta_i) (B-\theta_i) E_i(B-\theta_i)v.
\end{gather}
Since $E^{(2)}_i v=0$ the left-hand side of (\ref{AE^2v}) is zero. 
Since $(B-\theta_i)v\in V(\theta_i)$ it follows from Proposition \ref{prop:EKT}(iii) that 
$$
E_i(B-\theta_i)v\in V(\theta_{i-1}).
$$
Hence the right-hand side of (\ref{AE^2v}) is equal to the scalar multiple of $E_i(B-\theta_i)v$ by
\begin{gather}\label{e:scalar2}
((q^2+q^{-2})\theta_{i-1}-2\theta_i) (\theta_{i-1}-\theta_i).
\end{gather}
Since $\theta_i\not=\theta_{i-2}$ it follows from Lemma \ref{lem:theta^2} that $(q^2+q^{-2})\theta_{i-1}\not=2\theta_i$. Combined with $\theta_i\not=\theta_{i-1}$ the scalar (\ref{e:scalar2}) is nonzero. Therefore
\begin{gather*}\label{e:vinV(1)}
E_i(B-\theta_i)v=0.
\end{gather*}
Since $E_i|_{V(\theta_i)}$ is injective it follows that  
$
v\in V(\theta_i)$. 
The lemma follows.
\end{proof}

\begin{lem}\label{lem:cyclic_dimV(2)}
Suppose that   
$
E_i|_{V(\theta_i)}
$
is injective for each $i\in \Z/\dbar \Z$. Then the following hold:
\begin{enumerate}
\item If $\{\theta_i\}_{i\in \Z/\dbar \Z}$ is of type $O(2)$ or $O(-2)$ then 
$$
V^{(2)}(\theta_i)
\qquad 
\hbox{for all $i=1,2,\ldots,\displaystyle\frac{\dbar-1}{2}$}
$$ 
are of the same dimension.

\item If $\{\theta_i\}_{i\in \Z/\dbar \Z}$ is of type $E(2)$ then 
$$
V^{(2)}(\theta_i)
\qquad  
\hbox{for all $\displaystyle i=1,2,\ldots,\frac{\dbar}{2}-1$}
$$ 
are of the same dimension.

\item If $\{\theta_i\}_{i\in \Z/\dbar \Z}$ is of type $E(q+q^{-1})$ then 
$$
V^{(2)}(\theta_i) 
\qquad 
\hbox{for all $\displaystyle i=1,2,\ldots,\frac{\dbar}{2}$}
$$ 
are of the same dimension.
\end{enumerate}
\end{lem}
\begin{proof}
By Proposition \ref{prop:EKT2}(iii) we have the map 
\begin{gather}\label{E2}
E^{(2)}_i|_{V^{(2)}(\theta_i)}:V^{(2)}(\theta_i)\to V^{(2)}(\theta_{i-1})
\end{gather}
for each $i\in \Z/\dbar \Z$. 

(i): By Definition \ref{defn:DEO}(ii), (iii) the scalar $\theta_i\in \{\pm 2\}$ if and only if $i=0$. Combined with Lemma \ref{lem:theta^2} this yields that $\theta_i=\theta_{i-2}$ if and only if $i=1$. It follows from Proposition \ref{prop:multi_DEO}(ii) that $\theta_i=\theta_{i-1}$ if and only if  $i=\frac{\dbar+1}{2}$. Therefore the map (\ref{E2}) is injective for each  
$i\in \Z/\dbar\Z\setminus\{0,1,\frac{\dbar+1}{2}\}$ by Lemma \ref{lem:cyclic_E(2)}.
It follows that  
\begin{align} 
&\dim V^{(2)}(\theta_{\frac{\dbar-1}{2}})\leq \dim V^{(2)}(\theta_{\frac{\dbar-3}{2}})\leq \ldots \leq \dim V^{(2)}(\theta_1);
\label{ineq:iv-1}
\\
&\dim V^{(2)}(\theta_{\dbar-1})
\leq 
\dim V^{(2)}(\theta_{\dbar-2})
\leq 
\ldots
\leq 
\dim V^{(2)}(\theta_{\frac{\dbar+1}{2}}). 
\label{ineq:iv-2}
\end{align}
By Proposition \ref{prop:multi_DEO}(ii) we have $\theta_i=\theta_{\dbar-i}$ for $i=1,2,\ldots,\frac{\dbar-1}{2}$.
Combined with
(\ref{ineq:iv-1}) and (\ref{ineq:iv-2}) this yields (i).

(ii): By Definition \ref{defn:DEO}(iv) the scalar $\theta_i\in \{\pm 2\}$ if and only if $i\in\{0,\frac{\dbar}{2}\}$. Combined with Lemma \ref{lem:theta^2} this yields that  $\theta_i=\theta_{i-2}$ if and only if $i\in\{1,\frac{\dbar}{2}+1\}$. By Proposition \ref{prop:multi_DEO}(ii) none of $i\in \Z/\dbar \Z$ satisfies $\theta_i=\theta_{i-1}$.
Therefore (\ref{E2}) is injective for each $i\in \Z/\dbar \Z\setminus\{0,1,\frac{\dbar}{2},\frac{\dbar}{2}+1\}$ by Lemma \ref{lem:cyclic_E(2)}. It follows that 
\begin{align}
&\dim V^{(2)}(\theta_{\frac{\dbar}{2}-1})\leq \dim V^{(2)}(\theta_{\frac{\dbar}{2}-2})\leq \ldots \leq \dim V^{(2)}(\theta_1);
\label{ineq:ii-1}
\\
&\dim V^{(2)}(\theta_{\dbar-1})
\leq 
\dim V^{(2)}(\theta_{\dbar-2})
\leq 
\ldots
\leq 
\dim V^{(2)}(\theta_{\frac{\dbar}{2}+1}). 
\label{ineq:ii-2}
\end{align}
By Proposition \ref{prop:multi_DEO}(ii) we have $\theta_i=\theta_{\dbar-i}$ for all $i=1,2,\ldots,\frac{\dbar}{2}-1$.
Combined with (\ref{ineq:ii-1}) and (\ref{ineq:ii-2}) this yields (ii).

(iii): By Definition \ref{defn:DEO}(v) none of the scalars $\{\theta_i\}_{i\in \Z/\dbar \Z}$ is in $\{\pm 2\}$. Combined with Lemma \ref{lem:theta^2} none of $\{\theta_i\}_{i\in \Z/\dbar \Z}$ satisfies $\theta_i=\theta_{i-2}$. Using Proposition \ref{prop:multi_DEO}(iii) yields that $\theta_i=\theta_{i-1}$ if and only if $i\in\{1,\frac{\dbar}{2}+1\}$.
Therefore (\ref{E2}) is injective for each 
$i\in \Z/\dbar\Z\setminus\{1,\frac{\dbar}{2}+1\}$ by Lemma \ref{lem:cyclic_E(2)}.
It follows that 
\begin{align} 
&\dim V^{(2)}(\theta_{\frac{\dbar}{2}})\leq \dim V^{(2)}(\theta_{\frac{\dbar}{2}-1})\leq \ldots \leq \dim V^{(2)}(\theta_1);
\label{ineq:iii-1}
\\
&\dim V^{(2)}(\theta_{\dbar})
\leq 
\dim V^{(2)}(\theta_{\dbar-1})
\leq 
\ldots
\leq 
\dim V^{(2)}(\theta_{\frac{\dbar}{2}+1}). 
\label{ineq:iii-2}
\end{align}
By Proposition \ref{prop:multi_DEO}(iii) we have $\theta_i=\theta_{\dbar+1-i}$ for all $i=1,2,\ldots,\frac{\dbar}{2}$. 
Combined with 
(\ref{ineq:iii-1}) and (\ref{ineq:iii-2}) this implies (iii).
\end{proof}

\begin{lem}\label{lem:dimV=dV(1)}
If
$
E_i|_{V(\theta_i)}
$
is injective for each $i\in \Z/\dbar \Z$ then 
\begin{gather}\label{dimV=dV(1)}
\dim V=\dbar\cdot \dim V(\theta_1).
\end{gather} 
\end{lem}
\begin{proof}
To see this we break the argument into four cases: (i) $\{\theta_i\}_{i\in \Z/\dbar \Z}$ is of type $D$; (ii) $\{\theta_i\}_{i\in \Z/\dbar \Z}$ is of type $O(2)$ or $O(-2)$; (iii) $\{\theta_i\}_{i\in \Z/\dbar \Z}$ is of type $E(2)$; (iv) $\{\theta_i\}_{i\in \Z/\dbar \Z}$ is of type $E(q+q^{-1})$.

(i): It follows from Proposition \ref{prop:decomposition}(i) that 
$$
\dim V=\sum_{i\in \Z/\dbar\Z}\dim V(\theta_i).
$$
Combined with Lemma \ref{lem:cyclic_V(1)} this yields (\ref{dimV=dV(1)}).

(ii): It follows from Proposition \ref{prop:decomposition}(ii) that 
$$
\dim V=\dim V(\theta_0)+\sum_{i=1}^{\frac{\dbar-1}{2}}\dim V^{(2)}(\theta_i).
$$
Combined with Lemmas \ref{lem:cyclic_V(1)} and \ref{lem:cyclic_dimV(2)=2dimV(1)} along with Lemma \ref{lem:cyclic_dimV(2)}(i), this yields (\ref{dimV=dV(1)}).

(iii), (iv): Similar to the proof of (ii). 
\end{proof}

\begin{thm}\label{thm:dim_cyclic}
If
$
E_i|_{V(\theta_i)}
$
is injective for each $i\in \Z/\dbar \Z$ then 
$V$ is of dimension $\dbar$.
\end{thm}
\begin{proof}
Since $3n^2-3n+1<4n^2$ for all $n\geq 1$ it follows from Proposition \ref{prop:<2d} that 
$\dim V<2 \dbar$. Combined with Lemma \ref{lem:dimV=dV(1)} this forces that $\dim V=\dbar$. The result follows.
\end{proof}

\noindent {\it Proof of Theorem \ref{thm:dimension}.} Combine Theorems \ref{thm:dim_bidiag} and \ref{thm:dim_cyclic}.

\subsection*{Acknowledgements}
The research is supported by the Ministry of Science and Technology of Taiwan under the project MOST 106-2628-M-008-001-MY4. Part of the research was done when the author was visiting International Christian University from January 16 to February 2 in 2018. He would like to thank Prof. Hiroshi Suzuki for his hospitality.

\bibliographystyle{amsplain}
\bibliography{MP}

\end{document}